\documentclass[a4paper]{amsart}
\usepackage{graphicx}
\usepackage{amssymb}
\usepackage{amsmath}
\usepackage{amsthm}
\usepackage{amscd}
\usepackage[all,2cell]{xy}

\UseAllTwocells \SilentMatrices
\newtheorem{thm}{Theorem}[section]

\newtheorem{cor}[thm]{Corollary}
\newtheorem{lem}[thm]{Lemma}
\newtheorem{exm}[thm]{Example}

\newtheorem{prop}[thm]{Proposition}
\theoremstyle{definition}

\theoremstyle{remark}
\newtheorem{rem}[thm]{\bf Remark}
\numberwithin{equation}{section}

\begin{document}
\title[The singularity category of a quadratic monomial algebra]
{The singularity category of a quadratic monomial algebra}

\author[Xiao-Wu Chen] {Xiao-Wu Chen}

\thanks{}
\date{\today}

\thanks{E-mail:xwchen$\symbol{64}$mail.ustc.edu.cn}

\maketitle

\dedicatory{}%
\commby{}%

\begin{abstract}
We exploit singular equivalences between artin algebras, that are induced from certain functors between  the stable module categories. Such functors are called pre-triangle equivalences. We construct two pre-triangle equivalences connecting the stable module category over a  quadratic monomial algebra and the one over an algebra with radical square zero. Consequently, we obtain an explicit singular equivalence between the two algebras.
\end{abstract}

\section{Introduction}
Let $A$ be an artin algebra. The singularity category $\mathbf{D}_{\rm sg}(A)$ of $A$ is introduced in \cite{Buc} under the name ``the stable derived category". The terminology is justified by the following fact: the algebra $A$ has finite global dimension if and only if the singularity category $\mathbf{D}_{\rm sg}(A)$ is trivial. Hence, the singularity category provides a homological invariant for algebras of infinite global dimension.

The singularity category captures the stable homological property of an algebra.¡¡ More precisely, certain information of the syzygy endofunctor on the stable $A$-module category is encoded in $\mathbf{D}_{\rm sg}(A)$. Indeed, as  observed in \cite{KV}, the singularity category is equivalent to the stabilization of the pair, that consists of  the stable module category and the syzygy endofunctor on it; see also \cite{Bel2000}. This fact is used in \cite{Chen2011} to describe the singularity category of an algebra with radical square zero. We mention that a similar argument appears in \cite{Shen}.

By the fundamental result in \cite{Buc}, the stable category of Gorenstein-projective $A$-modules might be viewed as a triangulated subcategory of $\mathbf{D}_{\rm sg}(A)$. Moreover, if the algebra $A$ is Gorenstein, the two categories are triangle equivalent. We mention that the study of Gorenstein-projective modules goes back to \cite{ABr} under the name ``modules of G-dimension zero".  The Verdier quotient category $\mathbf{D}_{\rm def}(A)$ of the singularity category $\mathbf{D}_{\rm sg}(A)$  by the stable category of Gorenstein-projective $A$-modules is called the Gorenstein defect category of $A$  in \cite{BJO}. This terminology is justified by the fact that the algebra $A$ is Gorenstein if and only if the category $\mathbf{D}_{\rm def}(A)$ is trivial. In other words, the Gorenstein defect category measures how far the algebra is from being Gorenstein.

By a singular equivalence between two algebras, we mean a triangle equivalence between their singularity categories. We observe that a derived equivalence implies a singular equivalence. However, the converse is not true; for such examples, see \cite{Ch09,PSS}.  In general, a singular equivalence does not induce a triangle equivalence between Gorenstein defect categories. We mention a recent progress in \cite{ZZ}.

The aim of this paper is to study the singularity category of a quadratic monomial algebra. The main ingredient is the following observation: for two algebras, a certain functor between their stable module categories induces a singular equivalence after the stabilization. We call such a functor a pre-triangle equivalence between the stable module categories. More generally, the two stable module categories are called pre-triangle quasi-equivalent provided that there is a zigzag of pre-triangle equivalences connecting them. In this case, we also have a singular equivalence.

 The main result claims a pre-triangle quasi-equivalence between the stable module category of a quadratic monomial algebra and the one of an algebra with radical square zero; see Theorem \ref{thm:main}.  Combining this with the results in \cite{Chen2011,Sm,CYang}, we describe the singularity category of a quadratic monomial algebra via the category of finitely generated graded projective modules over the Leavitt path algebra of a certain quiver; see Proposition \ref{prop:final}. We mention that this description extends the result in \cite{Kal} on the singularity category of a gentle algebra; see also \cite{CGL,CSZ}.

The paper is organized as follows. In Section 2, we recall the stabilization of a looped category. We introduce the notion of a pre-stable equivalence between looped categories, which is a functor between looped categories that induces an equivalence after the stabilization. A pre-stable equivalence in the left triangulated case is called a pre-triangle equivalence, which induces a triangle equivalence after the stabilization. In Section 3, we recall the result in \cite{KV} which claims that the singularity category of an algebra is triangle equivalent to the stabilization of the stable module category. Therefore, a pre-triangle equivalence between stable module categories induces a singular equivalence; see Proposition \ref{prop:key} and compare Proposition \ref{prop:key-inv}. We include explicit examples of pre-triangle equivalences between stable module categories.

 In Section 4, we associate to a quadratic monomial algebra $A$ an algebra $B$ with radical square zero; compare \cite{CSZ}. We construct explicitly two pre-triangle equivalences connecting the stable $A$-module category  to the stable $B$-module category. Then we obtain the required singular equivalence between $A$ and $B$; see Theorem \ref{thm:main}. In Section 5, we combine Theorem \ref{thm:main} with the results in \cite{Chen2011,Sm,CYang} on the singularity category of an algebra with radical square zero. We describe the singularity category and the Gorenstein defect category of a quadratic monomial algebra via the categories of finitely generated graded projective modules over Leavitt path algebras of certain quivers; see Proposition \ref{prop:final}. We discuss some concrete examples at the end.

\section{The stabilization of a looped category}

In this section, we recall the construction of the stabilization of a looped category. The basic references are \cite[Chapter I]{Hel}, \cite[\S 1]{Tier}, \cite{KV} and \cite[Section 3]{Bel2000}.

Following \cite{Bel2000}, a \emph{looped category}$(\mathcal{C}, \Omega)$ consists of a category $\mathcal{C}$ with an endofunctor $\Omega\colon \mathcal{C}\rightarrow \mathcal{C}$, called the \emph{loop functor}. The looped category $(\mathcal{C}, \Omega)$ is said to be \emph{stable} if the loop functor $\Omega$ is an auto-equivalence on $\mathcal{C}$, while it is \emph{strictly stable} if $\Omega$ is an automorphism.

  By a \emph{looped functor} $(F, \delta)$ between two looped categories $(\mathcal{C}, \Omega)$ and $(\mathcal{D}, \Delta)$, we mean a functor $F\colon \mathcal{C}\rightarrow \mathcal{D}$ together with a natural isomorphism $\delta\colon F\Omega\rightarrow \Delta F$. For a looped functor $(F, \delta)$, we define  inductively  for each $i\geq 1$ a natural isomorphism $\delta^i\colon F\Omega^i\rightarrow \Delta^i F$ such that $\delta^1=\delta$ and $\delta^{i+1}=\Delta^i \delta \circ \delta^i\Omega$. Set $\delta^0$ to be the identity transformation on $F$.

We say that a looped functor $(F, \delta)\colon (\mathcal{C}, \Omega)\rightarrow (\mathcal{D}, \Delta)$ is  \emph{strictly looped} provided that $F\Omega=\Delta F$ as functors and $\delta$ is the identity transformation on $F\Omega$. In this case, we write $(F, \delta)$ as $F$; compare \cite[1.1]{Hel}.

Let $(\mathcal{C}, \Omega)$ be a looped category. We define a category $\mathcal{S}=\mathcal{S}(\mathcal{C}, \Omega)$ as follows. The objects of $\mathcal{S}$ are pairs $(X, n)$ with $X$ an object in $\mathcal{C}$ and $n\in \mathbb{Z}$. The Hom-set is defined by the following formula
\begin{align}
{\rm Hom}_\mathcal{S}((X, n), (Y, m))={\rm colim}\; {\rm Hom}_\mathcal{C}(\Omega^{i-n}(X), \Omega^{i-m}(Y)),
\end{align}
where $i$ runs over all integers satisfying $i\geq n$ and $i\geq m$. An element $f$ in ${\rm Hom}_\mathcal{S}((X, n), (Y, m))$ is said to have an \emph{$i$-th representative} $f_i\colon \Omega^{i-n}(X)\rightarrow \Omega^{i-m}(Y)$ provided that the canonical image of $f_i$ equals $f$.  The composition of morphisms in $\mathcal{S}$ is induced by the one in $\mathcal{C}$. We observe that $\tilde{\Omega}\colon \mathcal{S}\rightarrow \mathcal{S}$ sending $(X,n)$ to $(X, n-1)$ is an automorphism. Then we have a strictly stable category $(\mathcal{S}, \tilde{\Omega})$.

There is a canonical functor $\mathbf{S}\colon \mathcal{C}\rightarrow \mathcal{S}$ sending $X$ to $(X, 0)$, and a morphism $f$ to $\mathbf{S}(f)$ whose $0$-th representative is $f$. For an object $X$ in $\mathcal{C}$, we have a natural isomorphism $$\theta_X\colon (\Omega X, 0)\longrightarrow (X, -1),$$
whose $0$-th representative is ${\rm Id}_{\Omega X}$. Indeed, this yields a looped functor $$(\mathbf{S}, \theta)\colon (\mathcal{C}, \Omega)\longrightarrow (\mathcal{S}, \tilde{\Omega}).$$
 This process is called in \cite{Hel} the \emph{stabilization} of the looped functor $(\mathcal{C}, \Omega)$. We mention that $\mathbf{S}\colon \mathcal{C}\rightarrow \mathcal{S}$ is an equivalence if and only if $(\mathcal{C}, \Omega)$ is a stable category, in which case we identify $(\mathcal{C}, \Omega)$ with $(\mathcal{S}, \tilde{\Omega})$.

The stabilization functor $(\mathbf{S}, \theta)$ enjoys a universal property; see \cite[Proposition 1.1]{Hel}. Let $(F, \delta)\colon (\mathcal{C}, \Omega)\rightarrow (\mathcal{D}, \Delta)$ be a looped functor with $(\mathcal{D}, \Delta)$ a strictly stable category. We denote by $\Delta^{-1}$ the inverse of $\Delta$.  Then there is a unique functor $\tilde{F}\colon (\mathcal{S}, \tilde{\Omega})\rightarrow (\mathcal{D}, \Delta)$ which is strictly looped satisfying $F=\tilde{F}\mathbf{S}$ and $\delta=\tilde{F}\theta$. The functor $\tilde{F}$ sends $(X, n)$ to $\Delta^{-n}F(X)$. For a morphism $f\colon (X, n)\rightarrow (Y, m)$ whose $i$-th representative is given by $f_i\colon \Omega^{i-n}(X)\rightarrow \Omega^{i-m}(Y)$, we have
\begin{align}
\tilde{F}(f)=\Delta^{-i}((\delta^{i-m}_Y) \circ F(f_i)\circ (\delta^{i-n}_X)^{-1})\colon \Delta^{-n} F(X)\longrightarrow \Delta^{-m} F(Y).
\end{align}

\begin{lem}\label{lem:1}
Keep the notation as above. Then the functor $\tilde{F}\colon (\mathcal{S}, \tilde{\Omega})\rightarrow (\mathcal{D}, \Delta)$ is an equivalence if and only if the following conditions are satisfied:
\begin{enumerate}
\item  for any  morphism $g\colon FX\rightarrow FY$ in $\mathcal{D}$, there exist $i\geq 0$ and a morphism $f\colon \Omega^i X\rightarrow \Omega^i Y$ in $\mathcal{C}$ satisfying $\Delta^i(g)=\delta^i_Y\circ F(f)\circ (\delta^i_X)^{-1}$;
\item for any two morphisms $f, f'\colon X\rightarrow Y$ in $\mathcal{C}$ with $F(f)=F(f')$, there exists $i\geq 0$  such that $\Omega^i(f)=\Omega^i(f')$;
\item for any object $D$ in $\mathcal{D}$, there exist $i\geq 0$ and an object $X$ in $\mathcal{C}$ satisfying $\Delta^i(D)\simeq F(X)$.
\end{enumerate}
\end{lem}

\begin{proof}
Indeed, the above three conditions are equivalent to the statements that $\tilde{F}$ is full, faithful and dense, respectively. We refer to \cite[1.2 Proposition]{Tier} for the details and compare \cite[Proposition 3.4]{Bel2000}.
\end{proof}

We now apply Lemma \ref{lem:1} to a specific situation. Let $(F, \delta)\colon (\mathcal{C}, \Omega) \rightarrow (\mathcal{C}', \Omega')$ be a looped functor. Consider the composition
\begin{align}\label{equ:1}
(\mathcal{C}, \Omega) \stackrel{(F, \delta)}\longrightarrow (\mathcal{C}', \Omega') \stackrel{(\mathbf{S}, \theta)}\longrightarrow (\mathcal{S}(\mathcal{C}', \Omega'), \tilde{\Omega}').
\end{align}
By the universal property of the stabilization, there is a unique strictly looped functor $\mathcal{S}(F, \delta)\colon (\mathcal{S}(\mathcal{C}, \Omega), \tilde{\Omega}) \rightarrow (\mathcal{S}(\mathcal{C}', \Omega'), \tilde{\Omega}')$ making the following diagram commutative.
\[\xymatrix{
(\mathcal{C}, \Omega) \ar[rr]^{(F, \delta)} \ar[d]_{(\mathbf{S}, \theta)} && (\mathcal{C}, \Omega') \ar[d]^-{(\mathbf{S}, \theta)}\\
(\mathcal{S}(\mathcal{C}, \Omega), \tilde{\Omega}) \ar[rr]^-{ \mathcal{S}(F, \delta)} && (\mathcal{S}(\mathcal{C}', \Omega'), \tilde{\Omega}')
}\]
We call the functor $\mathcal{S}(F, \delta)$ the \emph{stabilization} of $(F, \delta)$.

\begin{prop}\label{prop:stab}
Let $(F, \delta)\colon (\mathcal{C}, \Omega) \rightarrow (\mathcal{C}', \Omega')$ be a looped functor. Then its stabilization $\mathcal{S}(F, \delta)$ is an equivalence if and only if the following conditions are satisfied:
\begin{enumerate}
\item[(S1)] for any  morphism $g\colon FX\rightarrow FY$ in $\mathcal{C}'$, there exist $i\geq 0$ and a morphism $f\colon \Omega^i X\rightarrow \Omega^i Y$ in $\mathcal{C}$ satisfying ${\Omega'}^i(g)=\delta^i_Y\circ F(f)\circ (\delta^i_X)^{-1}$;
\item[(S2)]for any two morphisms $f, f'\colon X\rightarrow Y$ in $\mathcal{C}$ with $F(f)=F(f')$, there exists $i\geq 0$  such that $\Omega^i(f)=\Omega^i(f')$;
\item[(S3)] for any object $C'$ in $\mathcal{C}'$, there exist $i\geq 0$ and an object $X$ in $\mathcal{C}$ satisfying ${\Omega'}^i(C')\simeq F(X)$.
\end{enumerate}
\end{prop}

The looped functor $(F, \delta)$ is called a \emph{pre-stable equivalence} if it satisfies (S1)-(S3). The result implies that a pre-stable equivalence induces an equivalence between the stabilized categories.

\begin{proof}
Write $(\mathcal{D}, \Delta)=(\mathcal{S}(\mathcal{C}', \Omega'), \tilde{\Omega}')$ and $\tilde{F}=\mathcal{S}(F, \delta)$. Write the composition (\ref{equ:1}) as $(\mathbf{S}F, \partial)$. Then for an object $X$ in $\mathcal{C}$, the morphism $\partial_X\colon \mathbf{S}F\Omega(X)\rightarrow \tilde{\Omega}'\mathbf{S}F(X)$ equals $\theta_{FX}\circ \mathbf{S}(\delta_X)$. We make the following observation: for a morphism $f\colon \Omega^l (X)\rightarrow \Omega^l (Y)$ in $\mathcal{C}$, the morphism $\partial_Y^l\circ \mathbf{S}F(f)\circ (\partial_X^{l})^{-1}$ has a $0$-th representative $\delta^{l}_Y\circ F(f)\circ (\delta^{l}_X)^{-1}$.

We claim that for each $1\leq i\leq 3$, the condition (S$i$) is equivalent to the condition ($i$) in Lemma \ref{lem:1}. Then we are done by Lemma \ref{lem:1}.

For the claim, we only prove that (S$i$) implies ($i$). By reversing the argument, we obtain the converse implication.

 For (1), we take a morphism $g\colon \mathbf{S}F(X)=(FX, 0)\rightarrow \mathbf{S}F(Y)=(FY, 0)$ in $\mathcal{D}$. We assume that $g$ has a $j$-th representative $g_j\colon {\Omega'}^jF(X)\rightarrow {\Omega'}^jF(Y)$. Consider the morphism $h\colon F(\Omega^j X)\rightarrow F(\Omega^j Y)$ by $h=(\delta^j_Y)^{-1}\circ g_j\circ \delta^j_X$. Then by (S$1$) there exist $i\geq 0$ and a morphism $f\colon \Omega^{i+j}(X)\rightarrow \Omega^{i+j}(Y)$ satisfying ${\Omega'}^i(h)=(\delta^i_{\Omega^jY})\circ F(f)\circ (\delta^i_{\Omega^jX})^{-1}$. Then we have $\Delta^{i+j}(g)=\partial^{i+j}_Y\circ \mathbf{S}F(f)\circ (\partial^{i+j}_X)^{-1}$. Here, we use the observation above and the fact that $\Delta^{i+j}(g)$ has a $0$-th representative $\Omega'^{i}(g_j)$.  We are done with (1).

For (2), we take two morphisms $f, f'\colon X\rightarrow Y$ in $\mathcal{C}$ with $\mathbf{S}F(f)=\mathbf{S}F(f')$. Then there exists $j\geq 0$ such that ${\Omega'}^jF(f)={\Omega'}^jF(f')$. Using the natural isomorphism $\delta^j$, we infer that $F\Omega^j(f)=F\Omega^j(f')$. By (S2) there exists $i\geq 0$ such that $\Omega^{i+j}(f)=\Omega^{i+j}(f')$, proving (2).

For (3), we take any object $(C', n)$ in $(\mathcal{D}, \Delta)$. We may assume that $n\geq 0$. Otherwise, we use the isomorphism $\theta_{C'}^{-n}\colon ((\Omega')^{-n}(C'), 0)\simeq (C', n)$. By (S3) there exist $j\geq 0$ and an object $X$ in $\mathcal{C}$ satisfying ${\Omega'}^j(C')\simeq F(X)$. We observe that $\Delta^{j+n}(C', n)=(C', -j)$, which is isomorphic to $\mathbf{S}{\Omega'}^j(C')$, which is further isomorphic to $\mathbf{S}F(X)$. Set $i=j+n$. Then we have the required isomorphism $\Delta^i(C', n)\simeq \mathbf{S}F(X)$ for (3).
\end{proof}

We make an easy observation.

\begin{cor}\label{cor:S3}
Let $(F, \delta)\colon (\mathcal{C}, \Omega) \rightarrow (\mathcal{C}', \Omega')$ be a looped functor. Assume that $F$ is fully faithful. Then $(F, \delta)$ is a pre-stable equivalence if and only if (S$3$) holds.
\end{cor}

\begin{proof}
By the fully-faithfulness of $F$, the conditions (S$1$) and (S$2$) hold trivially. We just take $i=0$ in both the conditions.
\end{proof}

We say that two looped categories $(\mathcal{C}, \Omega)$ and $(\mathcal{C}', \Omega')$ are \emph{pre-stably quasi-equivalent} provided that there exists a chain of looped categories \begin{align}\label{equ:2}
(\mathcal{C}, \Omega)=(\mathcal{C}_1, \Omega_1), (\mathcal{C}_2, \Omega_2), \cdots, (\mathcal{C}_n, \Omega_n)=(\mathcal{C}', \Omega')
\end{align}
such that for each $1\leq i\leq n-1$, there exists a pre-stable equivalence from $(\mathcal{C}_i, \Omega_i)$ to $(\mathcal{C}_{i+1}, \Omega_{i+1})$, or a pre-stable equivalence from $(\mathcal{C}_{i+1}, \Omega_{i+1})$ to $(\mathcal{C}_i, \Omega_i)$.

We have the following immediate consequence of Proposition \ref{prop:stab}.

\begin{cor}
Let $(\mathcal{C}, \Omega)$ and $(\mathcal{C}', \Omega')$ be two looped categories which are pre-stably quasi-equivalent. Then there is a looped functor $$(\mathcal{S}(\mathcal{C}, \Omega), \tilde{\Omega})\stackrel{\sim}\longrightarrow (\mathcal{S}(\mathcal{C}', \Omega'), \tilde{\Omega}'),$$
which is an equivalence.  \hfill $\square$
\end{cor}

Let $(F, \delta)\colon (\mathcal{C}, \Omega) \rightarrow (\mathcal{C}', \Omega')$ be a looped functor. A full subcategory $\mathcal{X}\subseteq \mathcal{C}$ is said to be \emph{saturated} provided that the following conditions are satisfied.
\begin{enumerate}
\item[(Sa1)] For each object $X$ in $\mathcal{C}$, there is a morphism $\eta_X\colon X\rightarrow G(X)$ with $G(X)$ in $\mathcal{X}$ such that $F(\eta_X)$ is an isomorphism and that $\Omega^d(\eta_X)$ is an isomorphism for some $d\geq 0$.
\item[(Sa2)] For a morphism $f\colon X\rightarrow Y$, there is a morphism $G(f)\colon G(X)\rightarrow G(Y)$ with $G(f)\circ \eta_X=\eta_Y\circ f$.
\item[(Sa3)] The conditions (S1)-(S3) above hold by requiring that all the objects $X, Y$ belong to $\mathcal{X}$.
\end{enumerate}

\begin{exm}
{\rm Let $(F, \delta)\colon (\mathcal{C}, \Omega) \rightarrow (\mathcal{C}', \Omega')$ be a looped functor. Assume that $F$ has a right adjoint functor $G$, which is fully faithful. Assume further that the unit $\eta\colon {\rm Id}_\mathcal{C}\rightarrow GF$ satisfies the following condition: for each object $X$, there exists $d\geq 0$ with $\Omega^d(\eta_X)$ an isomorphism. Take $\mathcal{X}$ to be the essential image of $G$.

We claim that $\mathcal{X}\subseteq \mathcal{C}$ is a saturated subcategory. Indeed, the restriction $F|_\mathcal{X}\colon \mathcal{X}\rightarrow \mathcal{C}'$ is an equivalence. Then (Sa3) holds trivially, by taking $i$ to be zero in (S1)-(S3). The conditions (Sa1) and (Sa2) are immediate from the assumption. Here, we use the well-known fact that $F(\eta)$ is a natural isomorphism, since $G$ is fully faithful.}
\end{exm}

\begin{lem}\label{lem:stab}
Let $(F, \delta)\colon (\mathcal{C}, \Omega) \rightarrow (\mathcal{C}', \Omega')$ be a looped functor, and $\mathcal{X}\subseteq \mathcal{C}$ a saturated subcategory. Then the conditions (S1)-(S3) hold, that is, the functor $(F, \delta)$ is a pre-stable equivalence.
\end{lem}

\begin{proof}
It suffices to verify (S1) and (S2). For (S1), take any morphism $g\colon FX\rightarrow FY$ in $\mathcal{D}$. Consider $g'=F(\eta_Y) \circ g\circ F(\eta_X)^{-1}\colon FG(X)\rightarrow FG(Y)$. Then by (Sa3) there exist $i\geq 0$  and $f'\colon \Omega^i (GX)\rightarrow \Omega^i (GY)$ with ${\Omega'}^i(g')=\delta^i_{GY}\circ F(f')\circ (\delta^i_{GX})^{-1}$. We may assume that $i$ is large enough such that both $\Omega^i(\eta_X)$ and $\Omega^i(\eta_Y)$ are isomorphisms. Take $f=(\Omega^i(\eta_Y))^{-1}\circ f'\circ \Omega^i(\eta_X)$, which is the required morphism in (S1).

Let $f, f'\colon X\rightarrow Y$ be morphisms with $F(f)=F(f')$. Applying (Sa2) and using the isomorphisms $F(\eta_X)$ and $F(\eta_Y)$, we have $FG(f)=FG(f')$. By (Sa3), we have ${\Omega}^iG(f)=\Omega^i G(f')$ for some $i\geq 0$. We assume that $i$ is large enough such that both $\Omega^i(\eta_X)$ and $\Omega^i(\eta_Y)$ are isomorphisms.  Then we infer from (Sa2) that $\Omega^i(f)=\Omega^i(f')$. We are done with (S2).
\end{proof}

We will specialize the consideration to left triangulated categories. A looped category $(\mathcal{C}, \Omega)$ is \emph{additive} provided that $\mathcal{C}$ is an additive category and the loop functor $\Omega$ is an additive functor. We recall that a  \emph{left triangulated category} $(\mathcal{C}, \Omega, \mathcal{E})$  consists of an additive looped category  $(\mathcal{C}, \Omega)$ and  a class $\mathcal{E}$ of left triangles in $\mathcal{C}$ satisfying certain axioms. If in addition the category $(\mathcal{C}, \Omega)$ is stable, we call $(\mathcal{C}, \Omega, \mathcal{E})$ a triangulated category, where the \emph{translation functor} $\Sigma$ of $\mathcal{C}$ is a quasi-inverse of $\Omega$. This notion is equivalent to the original one of a triangulated category in the sense of Verdier.  For details, we refer to \cite{BM} and compare \cite{KV}.

In what follows, we write $\mathcal{C}$ for the left triangulated category $(\mathcal{C}, \Omega, \mathcal{E})$. A looped functor $(F, \delta)$ between two left triangulated categories $\mathcal{C}$ and $\mathcal{C}'=(\mathcal{C}', \Omega', \mathcal{E}')$ is called a \emph{triangle functor} if $F$ is an additive functor and sends left triangles to left triangles. A triangle functor which is a pre-stable equivalence is called a \emph{pre-triangle equivalence}. Two left triangulated categories $\mathcal{C}$ and $\mathcal{C}'$ are \emph{pre-triangle quasi-equivalent} if they are pre-stably quasi-equivalent such that all  the categories in (\ref{equ:2}) are left triangulated  and all the pre-stable equivalences connecting them are pre-triangle equivalences.

 For a left triangulated category $\mathcal{C}=(\mathcal{C}, \Omega, \mathcal{E})$, the stabilized category $\mathcal{S}(\mathcal{C}):=(\mathcal{S}(\mathcal{C}, \Omega), \tilde{\Omega}, \tilde{\mathcal{E}})$ is a triangulated category, where the translation functor $\Sigma=(\tilde{\Omega})^{-1}$ and the triangles in  $\tilde{\mathcal{E}}$ are induced by the left triangles in $\mathcal{E}$; see \cite[Section 3]{Bel2000}.

\begin{cor}\label{cor:tri}
 Let $\mathcal{C}$ and $\mathcal{C}'$ be two left triangulated  categories which are pre-triangle  quasi-equivalent. Then there is a triangle equivalence $\mathcal{S}(\mathcal{C}) \stackrel{\sim}\longrightarrow \mathcal{S}(\mathcal{C}')$. \hfill $\square$
\end{cor}

\section{The singularity categories and singular equivalences}

In this section, we recall the notion of the singularity category of an algebra. We observe that for two algebras whose stable module categories are pre-triangle quasi-equivalent, their singularity categories are triangle equivalent; see Proposition \ref{prop:key} and compare Proposition \ref{prop:key-inv}.

Let $k$ be a commutative artinian ring with a unit. We emphasize that all the functors and categories are required to be $k$-linear in this section.

 Let $A$ be an artin $k$-algebra. We denote by $A\mbox{-mod}$ the category of finitely generated left $A$-modules, and by $A\mbox{-proj}$ the full subcategory consisting of projective modules.  We denote by $A\mbox{-\underline{mod}}$
the \emph{stable category} of $A\mbox{-mod}$ modulo projective modules (\cite[p.104]{ARS}). The morphism space $\underline{\rm Hom}_A(M, N)$ of two modules $M$ and $N$ in $A\mbox{-\underline{mod}}$ is defined to be ${\rm Hom}_A(M, N)/\mathbf{p}(M, N)$, where $\mathbf{p}(M, N)$ denotes the $k$-submodule formed by morphisms that factor through projective modules. For a morphism $f\colon M\rightarrow N$, we write $\bar{f}$ for its image in $\underline{\rm Hom}_A(M, N)$.

Recall that for an $A$-module $M$, its syzygy $\Omega_A(M)$ is the kernel of its projective cover $P(M)\stackrel{p_M}\rightarrow M$. We fix for $M$ a short exact sequence $0\rightarrow \Omega_A(M)\stackrel{i_M} \rightarrow P(M)\stackrel{p_M}\rightarrow M\rightarrow 0$. This gives rise to the \emph{syzygy functor} $\Omega_A\colon A\mbox{-\underline{mod}} \rightarrow A\mbox{-\underline{mod}}$; see \cite[p.124]{ARS}. Indeed, $A\mbox{-\underline{mod}}:=(A\mbox{-\underline{mod}}, \Omega_A, \mathcal{E}_A)$ is a left triangulated category, where $\mathcal{E}_A$ consists of  left triangles  that are induced from short exact sequences in $A\mbox{-mod}$. More precisely, given a short exact sequence $0\rightarrow X\stackrel{f}\rightarrow Y\stackrel{g}\rightarrow Z \rightarrow 0$, we have the following commutative diagram
\[\xymatrix{
0\ar[r] & \Omega_A(Z) \ar@{.>}[d]^{h} \ar[r]^-{i_Z} & P(Z) \ar@{.>}[d] \ar[r]^-{p_Z} & Z \ar[r] \ar@{=}[d] & 0\\
0\ar[r] & X \ar[r]^-{f} & Y \ar[r]^-{g} & Z \ar[r] & 0.
}\]
Then $\Omega_A(Z)\stackrel{\bar{h}} \rightarrow X \stackrel{\bar{f}} \rightarrow Y \stackrel{\bar{g}} \rightarrow Z$ is a left triangle in $\mathcal{E}_A$.  As recalled in Section 2, the stabilized category $\mathcal{S}(A\mbox{-\underline{mod}})$ is a triangulated category.

 There is a more well-known description of this stabilized category as the singularity category; see \cite{KV}. To recall this, we denote by $\mathbf{D}^b(A\mbox{-mod})$
the bounded derived category of $A\mbox{-mod}$.  We identify an $A$-module $M$ as the corresponding stalk complex concentrated at degree zero, which is also denoted by $M$.

Recall that a complex in $\mathbf{D}^b(A\mbox{-mod})$ is \emph{perfect} provided that it is isomorphic to a bounded complex consisting of projective modules. The full subcategory
consisting of perfect complexes is denoted by ${\rm perf}(A)$, which is a triangulated  subcategory of $\mathbf{D}^b(A\mbox{-mod})$ and is closed under direct summands; see
\cite[Lemma 1.2.1]{Buc}. Following \cite{Or04},  the  \emph{singularity category} of an algebra $A$ is defined to be the quotient triangulated category $\mathbf{D}_{\rm sg}(A)=\mathbf{D}^b(A\mbox{-mod})/{{\rm perf}(A)}$; compare \cite{Buc, KV, Hap91}. We denote by $q\colon \mathbf{D}^b(A\mbox{-mod})\rightarrow \mathbf{D}_{\rm sg}(A)$ the quotient functor.

We denote a complex of $A$-modules by $X^{\bullet}=(X^n, d^n)_{n\in \mathbb{Z}}$, where $X^n$ are
 $A$-modules and the differentials $d^n\colon X^n\rightarrow X^{n+1}$ are homomorphisms of modules satisfying $d^{n+1}\circ d^n=0$. The translation functor $\Sigma$ both on $\mathbf{D}^b(A\mbox{-mod})$ and $\mathbf{D}_{\rm sg}(A)$ sends  a complex $X^\bullet$ to a complex $\Sigma(X^\bullet)$, which is  given by $\Sigma(X)^n=X^{n+1}$ and $d_{\Sigma X}^n=-d^{n+1}_X$.

Consider the following functor
$$F_A\colon A\mbox{-\underline{mod}}\longrightarrow \mathbf{D}_{\rm sg}(A)$$ sending a module $M$ to the corresponding stalk complex concentrated at degree zero, and a morphism $\bar{f}$ to $q(f)$. Here, the well-definedness of $F_A$ on morphisms is due to the fact that a projective module is isomorphic to zero in $\mathbf{D}_{\rm sg}(A)$.

For an $A$-module $M$, we consider the two-term complex $C(M)=\cdots \rightarrow 0\rightarrow P(M)\stackrel{p_M}\rightarrow M\rightarrow 0\rightarrow \cdots$ with $P(M)$ at degree zero. Then we have a quasi-isomorphism $i_M\colon \Omega_A(M)\rightarrow C(M)$. The canonical inclusion ${\rm can}_M\colon \Sigma^{-1} M\rightarrow C(M)$ becomes an isomorphism in $\mathbf{D}_{\rm sg}(A)$. Then we have a natural  isomorphism
$$\delta_M=q({\rm can}_M)^{-1}\circ q(i_M)\colon F_A\Omega_A(M)\longrightarrow \Sigma^{-1}F_A(M).$$
In other words, $(F_A, \delta)\colon (A\mbox{-\underline{mod}}, \Omega_A)\rightarrow (\mathbf{D}_{\rm sg}(A), \Sigma^{-1}) $ is a looped functor. Indeed, $F_A$ is an additive functor and sends left triangles to (left) triangles. Then we have a triangle functor
$$(F_A, \delta)\colon A\mbox{-\underline{mod}} \longrightarrow \mathbf{D}_{\rm sg}(A).$$

Applying the universal property of the stabilization to $(F_A, \delta)$,  we obtain a strictly looped functor
$$\tilde{F}_A \colon \mathcal{S}(A\mbox{-\underline{mod}})\longrightarrow \mathbf{D}_{\rm sg}(A),$$
which is also a triangle functor; see \cite[3.1]{Bel2000}.

The following basic result is due to \cite{KV}. For a detailed proof, we refer to \cite[Corollary 3.9]{Bel2000}.

\begin{lem}\label{lem:KV}
Keep the notation as above. Then $\tilde{F}_A\colon \mathcal{S}(A\mbox{-\underline{\rm mod}})\rightarrow \mathbf{D}_{\rm sg}(A)$ is a triangle equivalence. \hfill $\square$
\end{lem}

By a \emph{singular equivalence} between two algebras $A$ and $B$,  we mean a triangle equivalence between their singularity categories.

\begin{prop}\label{prop:key}
Let $A$ and $B$ be two artin algebras. Assume that the stable categories $A\mbox{-\underline{\rm mod}}$ and $B\mbox{-\underline{\rm mod}}$ are pre-triangle quasi-equivalent. Then there is a singular equivalence between $A$ and $B$.
\end{prop}

\begin{proof}
We just combine Lemma \ref{lem:KV} and Corollary \ref{cor:tri}.
\end{proof}

In the following two examples,  pre-triangle equivalences between stable module categories are explicitly given.

\begin{exm}
{\rm Let $A$ and $B'$ be artin algebras, and let $_AM_{B'}$ be an $A$-$B'$-bimodule. Consider the upper triangular matrix algebra $B=\begin{pmatrix}A& M\\0 & B'\end{pmatrix}$. We recall that a left $B$-module is a column vector $\begin{pmatrix} X\\ Y\end{pmatrix}$, where $_AX$ and $_{B'}Y$ are a left $A$-module and a left $B'$-module with an $A$-module homomorphism $\phi\colon M\otimes_{B'} Y\rightarrow X$; compare \cite[III, Proposition 2.2]{ARS}.

Consider the natural full embedding $i\colon A\mbox{-mod}\rightarrow B\mbox{-mod}$, sending an $A$-module $X$ to $i(X)=\begin{pmatrix} X\\ 0\end{pmatrix}$. Since $i$ preserves projective modules and is exact, it commutes with taking the syzygies. Then we have the induced functor $i\colon A\mbox{-\underline{\rm mod}}\rightarrow  B\mbox{-\underline{\rm mod}}$, which is a triangle functor.

 We claim that the induced functor $i$ is a pre-triangle equivalence if and only if the algebra $B'$ has finite global dimension. In this case,   by Proposition \ref{prop:key} there is a triangle equivalence $\mathbf{D}_{\rm sg}(A)\stackrel{\sim}\longrightarrow \mathbf{D}_{\rm sg}(B)$; compare \cite[Theorem 4.1(1)]{Ch09}.

Indeed, the induced functor $i$ is fully faithful. By Corollary \ref{cor:S3}, we only need  consider the condition (S3). Then we are done by noting the following fact: for any $B$-module $\begin{pmatrix}X\\ Y\end{pmatrix}$ and $d\geq 0$, we have  $\Omega^d\begin{pmatrix}X\\ Y\end{pmatrix}= \begin{pmatrix}X'\\ \Omega_{B'}^d(Y)\end{pmatrix} $ for some $A$-module $X'$. In particular, if $\Omega_{B'}^d(Y)=0$, the $B$-module $\Omega^d\begin{pmatrix}X\\ Y\end{pmatrix}$ lies in the essential image of $i$.}
\end{exm}

The following example is somehow more difficult.

\begin{exm}
{\rm Let $A$ and $B'$ be artin algebras, and let $_{B'}N_{A}$ be an $A$-$B'$-bimodule. Consider the upper triangular matrix algebra $B=\begin{pmatrix}B'& N\\0 & A\end{pmatrix}$.
We assume that $B'$ has finite global dimension.

Consider the natural projection functor  $p\colon B\mbox{-mod}\rightarrow A\mbox{-mod}$, sending an $B$-module $\begin{pmatrix} X\\ Y\end{pmatrix}$ to the $A$-module $Y$. It is an exact functor which sends projective modules to projective modules. Then we have the induced functor $p\colon B\mbox{-\underline{\rm mod}}\rightarrow A\mbox{-\underline{\rm mod}}$, which is a triangle functor.

Take $\mathcal{X}$ to be the full subcategory of $ B\mbox{-\underline{\rm mod}}$ consisting of  modules of the form $\begin{pmatrix} 0 \\ Y \end{pmatrix}$. We claim that $\mathcal{X}$ is a saturated subcategory of $B\mbox{-\underline{\rm mod}}$. Then by Lemma \ref{lem:stab} the induced functor $p$ is a pre-triangle equivalence. Therefore, by Proposition \ref{prop:key} there is a triangle equivalence $\mathbf{D}_{\rm sg}(B)\stackrel{\sim}\longrightarrow \mathbf{D}_{\rm sg}(A)$; compare \cite[Theorem 4.1(2)]{Ch09}.

We now prove the claim. For a $B$-module $C=\begin{pmatrix} X\\ Y\end{pmatrix}$, we consider the projection $\eta_C\colon \begin{pmatrix} X\\ Y\end{pmatrix}\rightarrow G(C)=\begin{pmatrix} 0\\ Y\end{pmatrix}$. Since its kernel has finite projective dimension, it follows that $\Omega_B^d(\eta_C)$ is an isomorphism for $d$ large enough. We observe that $p(\eta_C)$ is an isomorphism. Then we have (Sa1).

 The conditions (Sa2) and (Sa3) are trivial. Here for (S2) in $\mathcal{X}$, we use the following fact:  if a morphism $f\colon Y\rightarrow Y'$ of $A$-module factors through a projective $A$-module $P$, then the morphism $\begin{pmatrix} 0 \\ f \end{pmatrix}\colon \begin{pmatrix} 0 \\ Y \end{pmatrix}\rightarrow \begin{pmatrix} 0 \\ Y' \end{pmatrix}$ of $B$-modules  factors though $\begin{pmatrix} 0 \\ P \end{pmatrix}$, which has finite projective dimension; consequently,  we have  $\Omega_B^d \begin{pmatrix} 0 \\ f \end{pmatrix}=0$ for $d$ large enough.
}
\end{exm}

 Let $M$ be a left $A$-module. Then $M^*={\rm Hom}_A(M, A)$ is a right $A$-module. Recall that an $A$-module $M$ is \emph{Gorenstein-projective} provided that there is an acyclic complex $P^\bullet$ of projective $A$-modules such that the Hom-complex $(P^\bullet)^*={\rm Hom}_A(P^\bullet, A)$ is still acyclic and that $M$ is isomorphic to a certain cocycle $Z^i(P^\bullet)$ of $P^\bullet$.

  We denote by $A\mbox{-Gproj}$ the full subcategory of $A\mbox{-mod}$ formed by Gorenstein-projective $A$-modules. We observe that $A\mbox{-proj}\subseteq A\mbox{-Gproj}$.  We recall that the full subcategory $A\mbox{-Gproj}\subseteq A\mbox{-mod}$ is closed under direct summands, kernels of epimorphisms and extensions; compare \cite[(3.11)]{ABr}. In particular, for a Gorenstein-projective $A$-module $M$ all its syzygies $\Omega_A^i(M)$ are Gorenstein-projective.

Since $A\mbox{-Gproj}\subseteq A\mbox{-mod}$  is closed under
extensions, it becomes naturally an exact category in the sense of
Quillen \cite{Qui73}. Moreover, it is a \emph{Frobenius category},
that is, it has enough (relatively) projective and enough
(relatively) injective objects, and the class of projective objects
coincides with the class of injective objects. In fact, the class of
 projective-injective objects in $A\mbox{-Gproj}$ equals
$A\mbox{-proj}$. For details, we compare \cite[Proposition 2.13]{Bel2000}.

We denote by $A\mbox{-\underline{Gproj}}$ the full subcategory of $A\mbox{-\underline{mod}}$ consisting of Gorenstein-projective $A$-modules. Then the syzygy functor
$\Omega_A$ restricts to an auto-equivalence $\Omega_A\colon
A\mbox{-\underline{Gproj}}\rightarrow A\mbox{-\underline{Gproj}}$. Moreover, the stable category $A\mbox{-\underline{Gproj}}$ becomes a triangulated category such that  the translation functor is given by a quasi-inverse of $\Omega_A$, and that the triangles are induced by short exact sequences in $A\mbox{-Gproj}$. These are consequences of a general result in  \cite[Chapter I.2]{Hap88}. The inclusion functor ${\rm inc}\colon A\mbox{-\underline{Gproj}}\rightarrow A\mbox{-\underline{mod}}$ is a triangle functor between left triangulated categories.
We consider the composite of triangle functors
$$G_A\colon A\mbox{-\underline{Gproj}} \stackrel{\rm inc}\longrightarrow A\mbox{-\underline{mod}} \stackrel{F_A}\longrightarrow \mathbf{D}_{\rm sg}(A).$$

Let $M, N$ be Gorenstein-projective $A$-modules. By the fully-faithfulness of the functor  $\Omega_A\colon
A\mbox{-\underline{Gproj}}\rightarrow A\mbox{-\underline{Gproj}}$, the natural map
$$\underline{\rm Hom}_A(M, N)\longrightarrow {\rm Hom}_{\mathcal{S}(A\mbox{-}\underline{\rm mod})} (M, N)$$
induced by the stabilization functor $\mathbf{S}\colon A\mbox{-}\underline{\rm mod} \rightarrow  \mathcal{S}(A\mbox{-}\underline{\rm mod})$ is an isomorphism.  We identify $\mathcal{S}(A\mbox{-\underline{mod}})$ with $\mathbf{D}_{\rm sg}(A)$ by Lemma \ref{lem:KV}. Then this isomorphism implies that the triangle functor $G_A$ is fully faithful; compare \cite[Theorem 4.1]{Buc} and \cite[Theorem 4.6]{Hap91}.

Recall from \cite{Buc, Hap91} that an artin algebra $A$ is \emph{Gorenstein} if the regular module $A$ has finite injective dimension on both sides. Indeed, the two injective dimensions equal. We mention that a selfinjective algebra is Gorenstein, where any module is Gorenstein-projective.

The following result is also known. As a consequence, for a selfinjective algebra $A$ the stable module category $A\mbox{-\underline{mod}}$ and $\mathbf{D}_{\rm sg}(A)$ are triangle equivalent; see \cite[Theorem 2.1]{Ric}.

\begin{lem}\label{lem:2}
Let $A$ be an artin algebra. Then the following statements are equivalent.
\begin{enumerate}
\item The algebra $A$ is Gorenstein.
\item The inclusion functor  ${\rm inc}\colon A\mbox{-\underline{\rm Gproj}}\rightarrow A\mbox{-\underline{\rm mod}}$ is a pre-triangle equivalence.
\item The functor $G_A\colon A\mbox{-\underline{\rm Gproj}}\rightarrow \mathbf{D}_{\rm sg}(A)$ is a triangle equivalence.
\end{enumerate}
\end{lem}

  \begin{proof}
  Recall that $A$ is Gorenstein if and only if for any module $X$, there exists $d\geq 0$ with $\Omega_A^d(X)$ Gorenstein-projective; see \cite{Hoshi}. The inclusion functor in (2) is fully faithful. By Corollary \ref{cor:S3} it is a pre-triangle equivalence if and only if the condition (S3) in $A\mbox{-\underline{mod}}$ is satisfied. Then the equivalence ``$(1)\Leftrightarrow(2)$" follows.

Since $\Omega_A\colon A\mbox{-\underline{\rm Gproj}}\rightarrow A\mbox{-\underline{\rm Gproj}}$ is an auto-equivalence, we identify $A\mbox{-\underline{\rm Gproj}}$ with its stabilization $\mathcal{S}(A\mbox{-\underline{\rm Gproj}})$. By Lemma \ref{lem:KV}, we identify  $\mathbf{D}_{\rm sg}(A)$ with $\mathcal{S}(A\mbox{-\underline{mod}})$. Then the functor $G_A$ is identified with the stabilization of the inclusion functor in (2). Then the equivalence ``$(2)\Leftrightarrow(3)$" follows from Proposition \ref{prop:stab}.
  \end{proof}

Recall from \cite{BJO} that the \emph{Gorenstein defect category} of an algebra $A$ is defined to be the Verdier quotient category $\mathbf{D}_{\rm def}(A)=\mathbf{D}_{\rm sg}(A)/{{\rm Im}\; G_A}$, where ${\rm Im}\; G_A$ denotes the essential image of the fully-faithful triangle functor $G_A$ and thus is a triangulated subcategory of $\mathbf{D}_{\rm sg}(A)$. By Lemma \ref{lem:2}(3), the algebra $A$ is Gorenstein if and only if $\mathbf{D}_{\rm def}(A)$ is trivial; see also \cite{BJO}.

The following observation implies that pre-triangle equivalences seem to be  ubiquitous in the study of singular equivalences; compare Proposition \ref{prop:key}.

\begin{prop}\label{prop:key-inv}
Let $A$ and $B$ be artin algebras. Assume that $B$ is a Gorenstein algebra and that there is a singular equivalence between $A$ and $B$. Then there is a pre-triangle equivalence from $A\mbox{-\underline{\rm mod}}$ to $B\mbox{-\underline{\rm mod}}$.
\end{prop}

\begin{proof}
Using the triangle equivalence $G_B$, we have a triangle equivalence $H\colon \mathbf{D}_{\rm sg}(A)\rightarrow B\mbox{-\underline{Gproj}}$. Then we have the following composite of triangle functors
$$F\colon A\mbox{-\underline{\rm mod}}\stackrel{F_A}\longrightarrow \mathbf{D}_{\rm sg}(A) \stackrel{H}\longrightarrow B\mbox{-\underline{Gproj}}\stackrel{\rm inc}\longrightarrow B\mbox{-\underline{mod}}.$$
We claim that $F$ is a pre-triangle equivalence.

 Indeed, the functor $F_A$ is a pre-triangle equivalence by Lemma \ref{lem:KV}, where we identify $\mathbf{D}_{\rm sg}(A)$ with its stabilization $\mathcal{S}(\mathbf{D}_{\rm sg}(A))$. The inclusion functor is a pre-triangle equivalence by Lemma \ref{lem:2}(2). Therefore,  all the three functors above are pre-triangle equivalences. Then as their composition, so is the functor $F$.
\end{proof}

\section{The singularity category of a quadratic monomial algebra}

In this section, we study the singularity category of a quadratic monomial algebra $A$. We consider the  algebra $B$ with radical square zero that is defined by  the relation quiver of $A$. The main result claims that there is a pre-triangle quasi-equivalence between the stable $A$-module category and the stable $B$-module category. Consequently, we obtain an explicit singular equivalence between $A$ and $B$.

Let $Q=(Q_0, Q_1; s,t)$  be a finite quiver, where $Q_0$ is the set of vertices, $Q_1$ the set of arrows, and  $s,t\colon Q_1\rightarrow Q_0$ are maps which assign to each arrow $\alpha$ its starting vertex $s(\alpha)$ and its terminating vertex $t(\alpha)$.  A path $p$ of length $n$ in $Q$ is a sequence $p=\alpha_n\cdots \alpha_2\alpha_1$ of arrows such that $s(\alpha_i)=t(\alpha_{i-1})$ for $2\leq i\leq n$; moreover, we define its starting vertex $s(p)=s(\alpha_1)$ and its terminating vertex $t(p)=t(\alpha_n)$.  We observe that a path of length one is just an arrow. To each vertex $i$, we associate a trivial path $e_i$ of length zero, and set $s(e_i)=i=t(e_i)$.

For two paths $p$ and $q$ with $s(p)=t(q)$, we write $pq$ for their concatenation. As convention, we have $p=pe_{s(p)}=e_{t(p)}p$. For two paths $p$ and $q$ in $Q$, we say that $q$ is a sub-path of $p$ provided that $p=p''qp'$ for some paths $p''$ and $p'$.

Let $k$ be a field. The path algebra $kQ$ of a finite quiver $Q$ is defined as follows. As a $k$-vector space, it has a basis given by all the paths in $Q$. For two paths $p$ and $q$, their multiplication is given by the concatenation $pq$ if $s(p)=t(q)$, and it is zero, otherwise.  The unit of $kQ$ equals  $\sum_{i\in Q_0}e_i$. Denote by $J$ the two-sided ideal of $kQ$ generated by arrows. Then $J^d$ is spanned by all the paths of length at least $d$ for each $d\geq 2$. A two-sided ideal $I$ of $kQ$ is \emph{admissible} provided that $J^d\subseteq I\subseteq J^2$ for some $d\geq 2$. In this case, the quotient algebra $A=kQ/I$ is finite-dimensional.

We recall that an admissible ideal $I$ of $kQ$ is quadratic monomial provided that it is generated by some paths of length  two. In this case, the quotient algebra $A=kQ/I$ is called a \emph{quadratic monomial algebra}. Observe that the algebra $A$ is with radical square zero if and only if $I=J^2$. We call $kQ/J^2$ the algebra with radical square zero \emph{defined} by the quiver $Q$.

In what follows, $A=kQ/I$ is a quadratic monomial algebra. We denote by $\mathbf{F}$ the set of paths of length two contained in $I$. Following \cite{Zim}, a path $p$ in $Q$ is \emph{nonzero} in $A$ provided that it does not belong to $I$, or equivalently, $p$ does not contain a sub-path in $\mathbf{F}$. In this case, we abuse the image $p+I$ in $A$ with $p$. The set of nonzero paths forms a $k$-basis for $A$. For a path $p$ in $I$, we write $p=0$ in $A$.

For a nonzero path $p$, we consider the left ideal $Ap$ generated by $p$, which has a $k$-basis given by the nonzero paths $q$ such that $q=q'p$ for some path $q'$. We observe that for a vertex $i$, $Ae_i$ is an indecomposable projective $A$-module.   Then we have a projective cover $\pi_p\colon Ae_{t(p)}\rightarrow Ap$ sending $e_{t(p)}$ to $p$.

\begin{lem}\label{lem:qmono}
Let $A=kQ/I$ be a quadratic monomial algebra. Then the following statements hold.
\begin{enumerate}
\item For a nonzero path $p=\alpha p'$ with $\alpha$ an arrow, there is an isomorphism $Ap\simeq A\alpha$ of $A$-modules sending $xp$ to $x\alpha$ for any path $x$ with $s(x)=t(p)$.
\item For an arrow $\alpha$, we have  a short exact sequence of $A$-modules
\begin{align}\label{equ:3}
0\longrightarrow \bigoplus_{\{\beta\in Q_1\; |\; \beta\alpha \in \mathbf{F}\}} A\beta \stackrel{\rm inc}\longrightarrow Ae_{t(\alpha)}\stackrel{\pi_\alpha}\longrightarrow A\alpha \longrightarrow 0,
\end{align}
where ``${\rm inc}$" denotes the inclusion map.
\item For any $A$-module $M$, there is  an isomorphism $\Omega_A^2(M)\simeq \bigoplus_{\alpha\in Q_1} (A\alpha)^{n_\alpha}$ for some integers $n_\alpha$.
\end{enumerate}
\end{lem}

 \begin{proof}
(1) is trivial and (2) is straightforward; compare the first paragraph in \cite[p.162]{Zim}. In view of (1), the last statement is a special case of \cite[Theorem I]{Zim}.
 \end{proof}

Let $\alpha$ be an arrow such that the set $\{\beta\in Q_1\; |\; \beta\alpha \in \mathbf{F}\}$ is nonempty. By (\ref{equ:3}), this is equivalent to the condition that the $A$-module $A\alpha$ is non-projective.  Denote by $N(\alpha)=\{\alpha'\in Q_1\; |\; t(\alpha')=t(\alpha), \; \beta\alpha'\in \mathbf{F} \mbox{ for each arrow }\beta \mbox{ satisfying } \beta\alpha\in \mathbf{F}\}$. Set $Z(\alpha)=\bigoplus_{\alpha'\in N(\alpha)} \alpha'A$, which is the right ideal generated by $N(\alpha)$. We observe that $\alpha\in N(\alpha)$.

The second statement of the  following result is analogous to \cite[Lemma 2.3]{CSZ}.

\begin{lem}\label{lem:sthom}
Let $\alpha, \alpha'$ be two arrows. We assume that the set $\{\beta\in Q_1\; |\; \beta\alpha \in \mathbf{F}\}$ is nonempty.   Then we have the following statements.
\begin{enumerate}
\item There is an isomorphism ${\rm Hom}_A(A\alpha, A)\rightarrow Z(\alpha)$ sending $f$ to $f(\alpha)$.
\item There is a $k$-linear isomorphism
\begin{align}
\underline{\rm Hom}_A(A\alpha, A\alpha')=\frac{Z(\alpha)\cap A\alpha'}{Z(\alpha)\alpha'}.
\end{align}
\item If $\alpha'$ does not belong to $N(\alpha)$, we have $\underline{\rm Hom}_A(A\alpha, A\alpha')=0$.
\item If $\alpha'$ belongs to $N(\alpha)$, there is a unique epimorphism  $\pi=\pi_{\alpha, \alpha'}\colon A\alpha \rightarrow A\alpha'$ sending $\alpha$ to $\alpha'$ and $\underline{\rm Hom}_A(A\alpha, A\alpha')=k\bar{\pi}$.
\end{enumerate}
\end{lem}

\begin{proof}
We observe that $Z(\alpha)$ has a $k$-basis given by nonzero paths $q$ which satisfy $t(q)=t(\alpha)$ and $\beta q=0$ for each arrow $\beta$ with $\beta\alpha\in \mathbf{F}$. Then we infer (1) by applying ${\rm Hom}_A(-, A)$ to (\ref{equ:3}) and using the canonical  isomorphism ${\rm Hom}_A(Ae_{t(\alpha)}, A)\simeq e_{t(\alpha)}A$.

For (2), we identify for each left ideal $K$ of $A$, ${\rm Hom}_A(A\alpha, K)$ with the subspace of ${\rm Hom}_A(A\alpha, A)$ formed by those morphisms whose image is contained in $K$. Therefore, we identify ${\rm Hom}_A(A\alpha, A\alpha')$ with $Z(\alpha)\cap A\alpha'$, ${\rm Hom}_A(A\alpha, Ae_{t(\alpha')})$ with $Z(\alpha)\cap Ae_{t(\alpha')}$. Recall the projective cover $\pi_{\alpha'}\colon Ae_{t(\alpha')}\rightarrow A\alpha'$. The subspace $\mathbf{p}(A\alpha, A\alpha')$ formed by those morphisms factoring through projective modules equals the image of the map ${\rm Hom}_A(\pi_{\alpha'}, A)$. This image is then identified with $Z(\alpha)\alpha'$. Then the required isomorphism follows.

(3) is an immediate consequence of (2), since in this case we have $Z(\alpha)\cap A\alpha'=Z(\alpha)\alpha'$.

For (4), we observe in this case that $Z(\alpha)\cap A\alpha'=(Z(\alpha)\alpha')\oplus k\alpha'$. It follows from (3) that  $\underline{\rm Hom}_A(A\alpha, A\alpha')$ is one dimensional. The existence of the surjective homomorphism $\pi$ is by the isomorphism in (1), under which $\pi$ corresponds to the element $\alpha'$. Then we are done.
\end{proof}

\begin{rem}\label{rem:1}
Assume that $\alpha'\in N(\alpha)$. In particular, $t(\alpha)=t(\alpha')$. Then we have the following commutative diagram
\[\xymatrix{
0\ar[r] & \bigoplus_{\{\beta\in Q_1\; |\; \beta\alpha \in \mathbf{F}\}} A\beta \ar[d]^-{{\rm inc}} \ar[r]^-{\rm inc} & Ae_{t(\alpha)} \ar@{=}[d] \ar[r]^-{\pi_\alpha} & A\alpha \ar[d]^-{\pi_{\alpha, \alpha'}} \ar[r] & 0\\
0\ar[r] & \bigoplus_{\{\beta\in Q_1\; |\; \beta\alpha' \in \mathbf{F}\}} A\beta \ar[r]^-{\rm inc} & Ae_{t(\alpha)} \ar[r]^-{\pi_{\alpha'}} & A\alpha' \ar[r] & 0.
}\]
The leftmost inclusion uses the fact that $\alpha'\in N(\alpha)$ and thus $\{\beta\in Q_1\; |\; \beta\alpha \in \mathbf{F}\}\subseteq \{\beta\in Q_1\; |\; \beta\alpha' \in \mathbf{F}\}$.
\end{rem}

The following notion is taken from \cite[Section 5]{CSZ}; compare \cite{HS}. Let $A=kQ/I$ be a quadratic monomial algebra. The \emph{relation quiver} $\mathcal{R}_A$ of $A$ is defined as follows. Its vertices are given by arrows in $Q$, and there is an arrow $[\beta \alpha]$ from $\alpha$ to $\beta$ for each element $\beta\alpha$ in $\mathbf{F}$. We consider the algebra $B=k\mathcal{R}_A/J^2$ with radical square zero defined by $\mathcal{R}_A$.

The main result of this paper is as follows.

\begin{thm}\label{thm:main}
Let $A=kQ/I$ be a quadratic monomial algebra, and let $B=k\mathcal{R}_A/J^2$ be the algebra with radical square zero defined by the relation quiver of $A$. Then there is a pre-triangle quasi-equivalence between $A\mbox{-\underline{\rm mod}}$ and $B\mbox{-\underline{\rm mod}}$. Consequently, there is a singular equivalence between $A$ and $B$.
\end{thm}

For an arrow $\alpha$ in $Q$, we denote by $S_\alpha$ and $P_\alpha$ the simple $B$-module and the indecomposable projective $B$-module corresponding to the vertex $\alpha$, respectively. We may identify $P_\alpha$ with $Be_\alpha$,  where $e_\alpha$ denotes the trivial path in $\mathcal{R}_A$ at $\alpha$. Hence, the $B$-module $P_\alpha$ has a $k$-basis $\{e_\alpha, [\beta\alpha]\; |\; \beta\alpha\in \mathbf{F}\}$.  We observe the following short exact sequence of $B$-modules
 \begin{align}\label{equ:4}
 0\longrightarrow \bigoplus_{\{\beta\in Q_1\; |\; \beta\alpha\in \mathbf{F}\}} S_\beta \stackrel{i_\alpha}\longrightarrow P_\alpha \longrightarrow S_\alpha\longrightarrow 0,
 \end{align}
 where $i_\alpha$ identifies $S_\beta$ with the $B$-submodule $k[\beta\alpha]$.

 We denote by $B\mbox{-\underline{ssmod}}$ the full subcategory of $B\mbox{-\underline{mod}}$ consisting of semisimple $B$-modules. We observe that for any $B$-module $M$, the syzygy $\Omega_B(M)$ is semisimple; compare \cite[Lemma 2.1]{Chen2012}. Moreover, any homomorphism $f\colon X\rightarrow Y$ between semisimple modules \emph{splits}, that is, it is isomorphic to a homomorphism of the form $\begin{pmatrix}0 & {\rm Id}_Z\\
 0& 0\end{pmatrix}\colon K\oplus Z\rightarrow  C\oplus Z$ for some $B$-modules $K$, $C$ and $Z$. We infer that $B\mbox{-\underline{ssmod}}\subseteq B\mbox{-\underline{mod}}$ is a left triangulated subcategory. Moreover, all left triangles inside  $B\mbox{-\underline{ssmod}}$ are direct sums of trivial ones.

 There is a unique $k$-linear functor $F\colon B\mbox{-\underline{ssmod}}\rightarrow A\mbox{-\underline{mod}}$ sending $S_\alpha$ to $A\alpha$ for each arrow  $\alpha$ in $Q$. Here, for the well-definedness of $F$, we use the following fact, which can be obtained by comparing (\ref{equ:3}) and (\ref{equ:4}): the simple $B$-module $S_\alpha$ is projective if and only if so is the $A$-module $A\alpha$.

 We have the following key observation.

 \begin{lem}\label{lem:key}
 The functor $F\colon B\mbox{-\underline{\rm ssmod}}\rightarrow A\mbox{-\underline{\rm mod}}$ is a pre-triangle equivalence.
 \end{lem}

 \begin{proof}
 Let $\alpha$ be an arrow in $Q$. We observe that (\ref{equ:3}) and (\ref{equ:4}) compute the syzygies modules $\Omega_A(A\alpha)$ and $\Omega_B(S_\alpha)$, respectively. It follows that the functor $F$ commutes with the syzygy functors. In other words, there is a natural isomorphism $\delta\colon F\Omega_A\rightarrow \Omega_BF$ such that $(F, \delta)$ is a looped functor. Since all morphisms in  $B\mbox{-\underline{\rm ssmod}}$ split, each left triangle inside is a direct sum of trivial ones. It follows that $F$ respects left triangles, that is, $(F, \delta)$ is a triangle functor.

 We verify the conditions (S1)-(S3) in Proposition \ref{prop:stab}. Then we are done. Since the functor $F$ is faithful, (S2) follows. The condition (S3) follows from Lemma  \ref{lem:qmono}(3).

 For (S1), we take a morphism $g\colon FX\rightarrow FY$ in $A\mbox{-\underline{mod}}$. Without loss of generality, we assume that both $X$ and $Y$ are indecomposable, in which case both are simple $B$-modules. We assume that $X=S_\alpha$ and $Y=S_{\alpha'}$. We assume that $g$ is nonzero, in particular, $FX=A\alpha$ is non-projective, or equivalently, the set $\{\beta\in Q_1\; |\; \beta\alpha\in \mathbf{F}\}$ is nonempty. Observe that $FY=A\alpha'$. We apply Lemma \ref{lem:sthom}(3) to infer that $\alpha'\in N(\alpha)$. Write $\pi=\pi_{\alpha, \alpha'}$. By Lemma \ref{lem:sthom}(4), we may assume that $g=\bar{\pi}$.

 The commutative diagram in Remark \ref{rem:1} implies that $\Omega_A(g)$ equals the inclusion morphism $\bigoplus_{\{\beta\in Q_1\; |\; \beta\alpha \in \mathbf{F}\}} A\beta\rightarrow \bigoplus_{\{\beta\in Q_1\; |\; \beta\alpha' \in \mathbf{F}\}} A\beta$. Take $f$ to be the corresponding inclusion morphism $\Omega_B(S_\alpha)=\bigoplus_{\{\beta\in Q_1\; |\; \beta\alpha \in \mathbf{F}\}} S_\beta\rightarrow \Omega_B(S_{\alpha'})=\bigoplus_{\{\beta\in Q_1\; |\; \beta\alpha' \in \mathbf{F}\}} S_\beta$ in $B\mbox{-\underline{ssmod}}$. Then we identify $F(f)$ with $\Omega_A(g)$, more precisely,  we have $F(f)=\delta_Y\circ \Omega_A(g)\circ (\delta_X)^{-1}$. This proves the condition (S1).
 \end{proof}

 We now prove Theorem \ref{thm:main}.
 \vskip 5pt

 \noindent \emph{Proof of Theorem \ref{thm:main}.}\quad Consider the inclusion functor ${\rm inc}\colon B\mbox{-\underline{ssmod}}\rightarrow B\mbox{-\underline{mod}}$. As mentioned above, this is a triangle functor. Recall that the syzygy of any $B$-module is semisimple, that is, it lies in $B\mbox{-\underline{ssmod}}$. Then the inclusion functor is a pre-triangle equivalence by Corollary \ref{cor:S3}. Recall the pre-triangle equivalence $F\colon B\mbox{-\underline{ssmod}}\rightarrow A\mbox{-\underline{mod}}$ in Lemma \ref{lem:key}. Then we have the required pre-triangle quasi-equivalence between $A\mbox{-\underline{mod}}$ and $B\mbox{-\underline{mod}}$.

  The last statement follows from Proposition \ref{prop:key}. We mention that by the explicit construction of the functor $F$, the resulting triangle equivalence $\mathbf{D}_{\rm sg}(A)\rightarrow \mathbf{D}_{\rm sg}(B)$ sends $A\alpha$ to $S_\alpha$ for each arrow $\alpha$ in $Q$.\hfill $\square$

 \section{Consequences and examples}

 In this section, we draw some consequences of Theorem \ref{thm:main} and describe some examples. For the Leavitt path algebra of a finite quiver, we refer to \cite{AA05,Sm,CYang}

 We first make some preparation by recalling some known results on the singularity category of an algebra with radical square zero. For a finite quiver $Q$, we denote by $kQ/J^2$ the corresponding algebra with radical square zero. Recall that a vertex in $Q$ is a sink if there is no arrow starting at it. We denote by $Q^0$ the quiver without sinks, that is obtained from $Q$ by repeatedly removing sinks.

 We denote by $L(Q)$ the \emph{Leavitt path algebra} of $Q$ with coefficients in $k$, which has a natural  $\mathbb{Z}$-grading. We denote by $L(Q)\mbox{-grproj}$ the category of finitely generated $\mathbb{Z}$-graded left $L(Q)$-modules, and by $(-1)\colon L(Q)\mbox{-grproj}\rightarrow L(Q)\mbox{-grproj}$ the degree-shift functor by degree $-1$.

For $n\geq 1$, we denote by $Z_n$ the \emph{basic $n$-cycle}, which is a connected quiver consisting of $n$ vertices and $n$ arrows which form an oriented cycle.  Then the algebra $kZ_n/J^2$ is selfinjective. In particular, the stable module category $kZ_n/J^2\mbox{-\underline{mod}}$ is triangle equivalent to $\mathbf{D}_{\rm sg}(kZ_n/J^2)$.

An abelian category  $\mathcal{A}$ is semisimple if any short exact sequence splits. For example, if the quiver $Q$ has no sinks, the category $L(Q)\mbox{-grproj}$ is a semisimple abelian category; see \cite[Lemma 4.1]{CYang}.

For a semisimple abelian category $\mathcal{A}$ and an auto-equivalence $\Sigma$ on $\mathcal{A}$, there is a unique triangulated structure on $\mathcal{A}$ with $\Sigma$ the translation functor. Indeed, all triangles are direct sums of trivial ones. The resulting triangulated category is denoted by $(\mathcal{A}, \Sigma)$; see \cite[Lemma 3.4]{Chen2011}.  As an example, we will consider the triangulated category $(L(Q)\mbox{-grproj}, (-1))$ for a quiver $Q$ without sinks.

\begin{exm}\label{exm:1}
{\rm  Let $k^n=k\times k\times \cdots \times k$ be the product algebra of $n$ copies of $k$. Consider the automorphism $\sigma\colon k^n\rightarrow k^n$ sending $(a_1, a_2, \cdots, a_n)$ to $(a_2, \cdots, a_n, a_1)$, which induces an automorphism $\sigma^*\colon k^n\mbox{-mod}\rightarrow k^n\mbox{-mod}$ by twisting  the $k^n$-action on modules. We observe that there are  triangle equivalences
$$(k^n\mbox{-mod}, \sigma^*)\stackrel{\sim}\longrightarrow  kZ_n/J^2\mbox{-\underline{mod}}\stackrel{\sim} \longrightarrow  (L(Z_n)\mbox{-grproj}, (-1)).$$
The first equivalence is well known and the second one is a special case of \cite[Theorem 6.1]{CYang}. We will denote all these triangulated categories by $\mathcal{T}_n$. }
\end{exm}

Let $Q$ be a finite quiver. We call a connected component $\mathcal{C}$ of $Q$ \emph{perfect} (resp. \emph{acyclic}) if it is a basic cycle (resp. it has no oriented cycles). A connected component is \emph{defect} if it is neither perfect nor acyclic. Then we have a disjoint union
$$Q=Q^{\rm perf}\cup Q^{\rm ac}\cup Q^{\rm def}$$
where $Q^{\rm perf}$ (resp. $Q^{\rm ac}$, $Q^{\rm def}$) is the union of all the perfect (resp. acyclic, defect) components in $Q$. Denote by $B=kQ/J^2$. Then we have a decomposition $B=B^{\rm perf}\times B^{\rm ac}\times B^{\rm def}$ of algebras.

We summarise the known results on the singularity category and the Gorenstein defect category of an algebra with radical square zero.

\begin{lem}\label{lem:radical2}
Keep the notation as above. Then the following statements hold.
\begin{enumerate}
\item There is a triangle equivalence $\mathbf{D}_{\rm sg}(B) \simeq B^{\rm perf}\mbox{-\underline{\rm mod}}\times \mathbf{D}_{\rm sg}(B^{\rm def})$.
\item There is a triangle equivalence $B\mbox{-\underline{\rm Gproj}} \simeq B^{\rm perf}\mbox{-\underline{\rm mod}}$, which is triangle equivalent to a product of categories $\mathcal{T}_n$.
    \item There is a triangle equivalence $\mathbf{D}_{\rm def}(B) \simeq \mathbf{D}_{\rm sg}(B^{\rm def})$, which is triangle equivalent to $(L((Q^{\rm def})^0)\mbox{-{\rm grproj}}, (-1))$.
\end{enumerate}
\end{lem}

\begin{proof}
We observe that the algebra $B^{\rm perf}$ is selfinjective and that $B^{\rm ac}$ has  finite global dimension. Then (1) is a consequence of the decomposition $\mathbf{D}_{\rm sg}(B)=\mathbf{D}_{\rm sg}(B^{\rm perf})\times \mathbf{D}_{\rm sg}(B^{\rm ac})\times \mathbf{D}_{\rm sg}(B^{\rm def})$ of categories.

For (2), we note that any $B^{\rm perf}$-module is Gorenstein-projective and that a Gorenstein-projective $B^{\rm ac}$-module is necessarily projective. By \cite[Theorem 1.1]{Chen2012} any Gorenstein-projective $B^{\rm def}$-module is projective. Then (2) follows by a similar decomposition of $B\mbox{-\underline{\rm Gproj}}$. The last statement follows from Example \ref{exm:1}, since $B^{\rm perf}$ is isomorphic to a product of algebras of the form $kZ_n/J^2$.

By (1) and (2), the functor $G_B\colon B\mbox{-\underline{\rm Gproj}}\rightarrow \mathbf{D}_{\rm sg}(B)$ is identified with the inclusion. The required triangle equivalence follows immediately. The last sentence follows by combining \cite[Proposition 4.2]{Chen2011} and  \cite[Theorem 6.1]{CYang}; compare \cite[Theorem B]{Chen2011} and \cite[Theorem 5.9]{Sm}.
 \end{proof}

In what follows, let $A=kQ/I$ be a quadratic monomial algebra with $\mathcal{R}_A$ its relation quiver. We denote by $\{\mathcal{C}_1, \mathcal{C}_2, \cdots, \mathcal{C}_m\}$ the set of all  the perfect components in $\mathcal{R}_A$, and by $d_i$ the number of vertices in the basic cycle $\mathcal{C}_i$.

Let $B=k\mathcal{R}_A/J^2$ be the algebra with radical square zero defined by $\mathcal{R}_A$. We consider the triangle equivalence $\Phi\colon \mathbf{D}_{\rm sg}(A)\rightarrow \mathbf{D}_{\rm sg}(B)$ obtained in Theorem \ref{thm:main}. We identify the fully faithful functors $G_A$ and $G_B$ as inclusions.

The following result describes the singularity category and the Gorenstein defect category of a quadratic monomial algebra. We mention that the equivalence in Proposition \ref{prop:final}(2) is due to \cite[Theorem 5.7]{CSZ}, which is obtained by a completely different method.

\begin{prop}\label{prop:final}
The triangle equivalence $\Phi\colon \mathbf{D}_{\rm sg}(A)\rightarrow \mathbf{D}_{\rm sg}(B)$ restricts to a triangle equivalence $A\mbox{-\underline{\rm Gproj}}\stackrel{\sim}\longrightarrow B\mbox{-\underline{\rm Gproj}}$, and thus induces a triangle equivalence $\mathbf{D}_{\rm def}(A)\stackrel{\sim}\longrightarrow \mathbf{D}_{\rm def}(B)$.

Consequently, we have the following triangle equivalences:
\begin{enumerate}
\item $\mathbf{D}_{\rm sg}(A)\stackrel{\sim}\longrightarrow A\mbox{-\underline{\rm Gproj}} \times \mathbf{D}_{\rm def}(A)$;
\item $A\mbox{-\underline{\rm Gproj}}\stackrel{\sim}\longrightarrow B^{\rm perf}\mbox{-\underline{\rm mod}}\stackrel{\sim}\longrightarrow \mathcal{T}_{d_1}\times \mathcal{T}_{d_2}\times \cdots \times \mathcal{T}_{d_m}$;
    \item $\mathbf{D}_{\rm def}(A) \stackrel{\sim}\longrightarrow \mathbf{D}_{\rm sg}(B^{\rm def})\stackrel{\sim}\longrightarrow (L(Q')\mbox{-{\rm grproj}}, (-1))$ with $Q'=(\mathcal{R}_A^{\rm def})^0$.
\end{enumerate}
\end{prop}

\begin{proof}
Recall from the proof of Theorem \ref{thm:main} that $\Phi(A\alpha)=S_\alpha$ for each arrow $\alpha$ in $Q$. By \cite[Lemma 5.4(1)]{CSZ} the $A$-module $A\alpha$ is non-projective Gorenstein-projective if and only if $\alpha$, as a vertex, lies in a perfect component of $\mathcal{R}_A$. Moreover, any indecomposable non-projective Gorenstein-projective $A$-module arises in this way. On the other hand, any indecomposable non-projective Gorenstein-projective $B$-module is of the form $S_\alpha$ with $\alpha$ in $\mathcal{R}_A^{\rm perf}$; see Lemma \ref{lem:radical2}(2). It follows that the equivalence $\Phi$ restricts to the equivalence $A\mbox{-\underline{\rm Gproj}}\stackrel{\sim}\longrightarrow B\mbox{-\underline{\rm Gproj}}$.

The three triangle equivalences follow immediately from the equivalences in Lemma \ref{lem:radical2}.
\end{proof}

We end the paper with  examples on Proposition \ref{prop:final}.

\begin{exm}
{\rm Let $A$ be a quadratic monomial algebra which is Gorenstein. By \cite[Proposition 5.5(1)]{CSZ} this is equivalent to the condition that the relation quiver $\mathcal{R}_A$ has no defect components. For example, a gentle algebra is such an example. Note that $\mathbf{D}_{\rm def}(A)$ is trivial.  Then we obtain a triangle equivalence
$$\mathbf{D}_{\rm sg}(A)\stackrel{\sim}\longrightarrow \mathcal{T}_{d_1}\times \mathcal{T}_{d_2}\times \cdots \times \mathcal{T}_{d_m},$$
 where $d_i$'s denote the sizes of the perfect components of $\mathcal{R}_A$. This result extends \cite[Theorem 2.5(b)]{Kal}; see also  \cite{CGL}.}
\end{exm}

\begin{exm}
{\rm Let $A=k\langle x, y\rangle/I$ be the quotient algebra of the free algebra $k\langle x, y\rangle$ by the ideal $I=(x^2, y^2, yx)$. Then the relation quiver $\mathcal{R}_A$ is as follows.
\[\xymatrix{
 x \ar@(dl,ul)^{[x^2]} \ar[rr]^-{[yx]}&& y \ar@(dr,ur)_{[y^2]}
}\]
The relation quiver has no perfect components. Then we have triangle equivalences $\mathbf{D}_{\rm sg}(A)\simeq \mathbf{D}_{\rm def}(A)\simeq (L(\mathcal{R}_A)\mbox{-{\rm grproj}}, (-1))$.}
\end{exm}

\begin{exm}
{\rm Consider the following quiver $Q$ and the algebra $A=kQ/I$ with $I=(\beta\alpha, \alpha\beta, \delta\gamma, \gamma\delta, \delta\xi)$.
\[\xymatrix{
1 \ar@/^/[r]^-\alpha & 2 \ar@/^/[l]^-{\beta} \ar@/^/[r]^-\gamma & 3 \ar@/^/[l]^-{\delta} & 4\ar[l]_{\xi}
}\]
Its relation quiver $\mathcal{R}_A$ is as follows.
\[\xymatrix{ \alpha \ar@/^/[r]^-{[\beta\alpha]} & \beta \ar@/^/[l]^-{[\alpha\beta]} &  \gamma \ar@/^/[r]^-{[\delta\gamma]} & \delta \ar@/^/[l]^-{[\gamma\delta]}  & \xi \ar[l]_{[\delta\xi]}}\]
There are one perfect component and one defect component; moreover, we observe  $(\mathcal{R}_A^{\rm def})^0=Z_2$. Then we have triangle equivalences $A\mbox{-\underline{\rm Gproj}}\simeq \mathcal{T}_2$ and $\mathbf{D}_{\rm def}(A)\simeq (L(Z_2)\mbox{-grproj}, (-1))$, which is equivalent to $\mathcal{T}_2$; see Example \ref{exm:1}. Therefore, we have a triangle equivalence $\mathbf{D}_{\rm sg}(A) \simeq \mathcal{T}_2\times \mathcal{T}_2$.}
\end{exm}

\vskip 10pt

\noindent {\bf Acknowledgements} \; The author thanks Dawei Shen and Dong Yang for helpful discussions. The work is supported by the National Natural Science Foundation of China (No.11201446), NECT-12-0507, and the Fundamental Research Funds for the Central Universities.

\bibliography{}

\vskip 10pt

 {\footnotesize \noindent Xiao-Wu Chen\\
 Key Laboratory of Wu Wen-Tsun Mathematics, Chinese Academy of Sciences\\
  School of Mathematical Sciences, University of Science and Technology of China\\
   No. 96 Jinzhai Road, Hefei, Anhui Province, 230026, P. R. China.\\
URL: http://home.ustc.edu.cn/$^\sim$xwchen}

\end{document}